\documentclass[letterpaper, 10 pt, conference]{ieeeconf}  

\IEEEoverridecommandlockouts                              

\usepackage{afterpage}
\usepackage{epsfig} 
\usepackage{times} 
\usepackage{amsmath} 
\usepackage{amssymb}  
\usepackage{amsfonts}
\usepackage{mathptmx} 
\usepackage{tabularx}
\usepackage{epstopdf}
\usepackage{subfigure}
\usepackage{float}
\usepackage{epic,color}
\usepackage{mathrsfs}
\usepackage{dsfont,MnSymbol}
\usepackage{algorithm}
\usepackage{algorithmic}
\usepackage{multirow} 
\usepackage[T1]{fontenc}
\newcounter{rmnum}
\newenvironment{romannum}{\begin{list}{{\upshape (\roman{rmnum})}}{\usecounter{rmnum}
\setlength{\leftmargin}{14pt}
\setlength{\rightmargin}{8pt}
\setlength{\itemsep}{2pt}
\setlength{\itemindent}{-1pt}
}}{\end{list}}

\newcounter{anum}

\newlength{\noteWidth}
\long\def\notes#1{\ifinner
             {\tiny #1}
             \else
              \marginpar{\parbox[t]{\noteWidth}{\raggedright\tiny #1}}
               \fi}

\def\IEEEQEDclosed{\mbox{\rule[0pt]{1.3ex}{1.3ex}}}

\def\qed{\ifmmode\IEEEQEDclosed\else{\unskip\nobreak\hfil
\penalty50\hskip1em\null\nobreak\hfil\IEEEQEDclosed
\parfillskip=0pt\finalhyphendemerits=0\endgraf}\fi}

\def\qed{\hspace*{\fill}~\IEEEQED\par\endtrivlist\unskip}

\def\Re{\mathbb{R}}
\def\R{\mathbb{R}}

\def\ind{\text{\rm\large 1}}

\def\Sec#1{Sec.~\ref{#1}}
\def\Fig#1{Fig.~\ref{#1}}

\def\notes#1{\marginpar{\tiny #1}\typeout{Notes!
Notes!
Notes!
}}
\renewcommand{\notes}[1]{\typeout{notes!}}

\def\ind{\bbbone}

\def\Re{\field{R}}
\def\k{{\sf K}}

\def\Sec#1{Sec.~\ref{#1}}

\def\clZ{{\cal Z}}

\def\Sec#1{Sec~\ref{#1}}

\def\E{{\sf E}}
\def\Expect{{\sf E}}

\def\Prob{{\sf P}}

\def\Expect{{\sf E}}

\def\R{\mathbb{R}}

\def\Fig#1{Fig.~\ref{#1}}
\def\Sec#1{Sec.~\ref{#1}}

\def\IEEEQEDclosed{\mbox{\rule[0pt]{1.3ex}{1.3ex}}}
\def\qed{\nobreak\hfill\IEEEQEDclosed}

\newcommand{\expect}{ {\sf E} }

\def\clZ{{\cal Z}}

\newtheorem{theorem}{Theorem}
\newtheorem{example}{Example}
\newtheorem{definition}{Definition}
\newtheorem{lemma}{Lemma}
\newtheorem{remark}{Remark}
\newtheorem{proposition}{Proposition}

\def\beq{\begin{eqnarray}} 
\def\bc{\begin{center}} 
\def\be{\begin{enumerate}}
\def\bi{\begin{itemize}} 
\def\bs{\begin{small}}
\def\bS{\begin{slide}}
\def\ec{\end{center}} 
\def\ee{\end{enumerate}}
\def\ei{\end{itemize}}
\def\es{\end{small}}
\def\eS{\end{slide}}
\def\eeq{\end{eqnarray}}

\def\hah{\hat{h}}

\def\hah{\hat{h}}

\newcommand{\newP}[1]{\medskip\noindent{\bf #1:}}

\newcommand{\ud}{\,\mathrm{d}}

\def\Re{\mathbb{R}}
\def\E{{\sf E}}

\def\ind{\text{\rm\large 1}}

\def\Sec#1{Sec.~\ref{#1}}

\def\Expect{{\sf E}}

\def\clZ{{\cal Z}}

\def\hah{\hat{h}}

\renewcommand{\Re}{\mathbb{R}}

\def\Prob{{\sf P}}

\newcommand{\keps}{{k}_{\epsilon}}
\newcommand{\kepsN}{{k}_{\epsilon}^{(N)}}
\newcommand{\Keps}{{K}_\epsilon}
\newcommand{\KepsN}{{K}_\epsilon^{(N)}}

\newcommand{\phiepsN}{{\phi}^{(N)}_\epsilon}
\newcommand{\phieps}{{\phi}_{\epsilon}}

\newcommand{\Teps}{{T}_{\epsilon}}

\newcommand{\TepsN}{{T}^{(N)}_\epsilon}

\newcommand{\geps}{{g}_{\epsilon}}
\newcommand{\pr}{\rho}

\newcommand{\vepsN}{{{\sf K}}_\epsilon^{(N)}}

\newcommand{\Geps}{{G}_\epsilon}
\newcommand{\GepsN}{{G}_\epsilon^{(N)}}
\newcommand{\Tm}{{\bf T}}

\newcommand{\hm}{{\bf h}}

\title{\LARGE \bf
Error Estimates for the Kernel Gain Function Approximation \\ in the Feedback
Particle Filter}

\author{Amirhossein Taghvaei, Prashant G. Mehta, Sean P. Meyn 
\thanks{Financial support from the NSF CMMI grants 1334987 and 1462773 is gratefully acknowledged. 
}
\thanks{A.~Taghvaei and P.~G.~Mehta are with the Coordinated
  Science Laboratory and the Department of Mechanical Science and
  Engineering at the University of Illinois at Urbana-Champaign (UIUC). S. P. Meyn is with the Department of Electrical and Computer Engineering at the University of Florida
{\tt\scriptsize taghvae2@illinois.edu; mehtapg@illinois.edu; meyn@ece.ufl.edu};}
}

\begin{document}

\maketitle
\thispagestyle{empty}
\pagestyle{empty}

\begin{abstract}

This paper is concerned with the analysis of the 
kernel-based algorithm for gain function approximation 
in the feedback particle filter.  The exact gain
function is the solution of a Poisson equation involving a
probability-weighted Laplacian.  The kernel-based method --
introduced in our prior work -- allows one
to approximate this solution using {\em only}  particles sampled from
the probability distribution.  This paper describes
new representations and algorithms based on the kernel-based method.
Theory surrounding the approximation is improved and a novel formula for the gain function
approximation is derived.  A procedure
for carrying out error analysis of the approximation is introduced.
Certain asymptotic estimates for bias and variance are derived
for the general nonlinear non-Gaussian case.  Comparison with the
constant gain function approximation is provided.
The results are 
illustrated with the aid of some numerical experiments.  

\end{abstract}
\section{Introduction}

This paper is concerned with the analysis of the kernel-based algorithm for numerical 
approximation of the gain function in the feedback particle filter algorithm; cf.~\cite{amirCDC2016}.  
The filter represents a numerical solution of the following continuous-time 
nonlinear filtering problem:
\begin{subequations}
\begin{flalign}
&\text{Signal:} \quad &\ud X_t &= a(X_t)\ud t + \ud B_t,\quad X_0\sim p_0^*,&
\label{eqn:Signal_Process}
\\
&\text{Observation:} \quad &\ud Z_t &= h(X_t)\ud t + \ud W_t,&
\label{eqn:Obs_Process}
\end{flalign}
\end{subequations}
where $X_t\in\Re^d$ is the (hidden) state at time $t$, the initial
condition $X_0$ has the prior density $p_0^*$, $Z_t \in\Re$ is the
observation, and $\{B_t\}$, $\{W_t\}$ are mutually independent
standard Wiener processes taking values in $\Re^d$ and $\Re$, respectively. The mappings 
$a(\cdot): \Re^d \rightarrow \Re^d$ and $h(\cdot): \Re^d \rightarrow
\Re$ are given $C^1$ functions.    The goal of the filtering problem is to approximate the posterior 
distribution of the state $X_t$ given the time history of observations (filtration) $\clZ_t :=
\sigma(Z_s:  0\le s \le t)$. 

\smallskip

The feedback particle filter (FPF) is a controlled stochastic differential equation (sde),
\begin{equation}
\begin{aligned}
&\text{FPF:}  \nonumber \\
&\ud X^i_t = a(X^i_t) \ud t + \ud B^i_t + \k_t(X^i_t) \circ (\ud Z_t -
\frac{h(X^i_t) + \hat{h}_t}{2}\ud t),
\;\; X_0^i \sim p_0^*, 
\end{aligned}
\label{eqn:particle_filter_nonlin_intro}
\end{equation}
for $i=1,\ldots,N$, where $X_t^i\in\Re^d$ is the state of the $i^\text{th}$ particle at time $t$, the initial condition $X^i_0\sim p_0^*$, $B^i_t$ is a standard Wiener process, and
$\hat{h}_t := \E[h(X_t^i)|\mathcal{Z}_t]$.  Both  $B^i_t$ and
$X^i_0$ are mutually independent and also independent of $X_t,Z_t$.  The $\circ$
indicates that the sde is expressed in its Stratonovich form.

\smallskip

The gain function $\k_t$ is obtained by solving a weighted Poisson
equation: For each fixed time $t$, the function $\phi$ is the solution to a Poisson equation,
\begin{flalign}
\label{eqn:EL_phi_intro}
\text{PDE:} && &
\begin{aligned}
\nabla \cdot (p(x,t) \nabla \phi(x,t) ) & = - (h(x)-\hat{h}) p(x,t),\\
\int \phi(x,t) p(x,t) \ud x & = 0 \qquad \text{(zero-mean)},
\end{aligned} &
\end{flalign}
where $\nabla$ and $\nabla \cdot $ denote the
gradient and the divergence operators, respectively, and
$p$ denotes the conditional density of $X_t^i$ given
$\mathcal{Z}_t$.  
In terms of the solution $\phi$, the gain function is given by,
\begin{flalign*}
&\text{Gain Function:}& \quad 
\k_t(x)&= \nabla \phi(x,t)\, . &
\label{eqn:gradient_gain_fn_intro}
\end{flalign*}
The gain function $\k_t$ is vector-valued (with dimension
$d\times 1$) and it needs to be obtained for each fixed time $t$.
For the linear Gaussian case, the gain function
is the Kalman gain.

\smallskip

FPF is an exact algorithm:  If the initial condition $X^i_0$ is
sampled from the prior $p_0^*$ then 
\[
\Prob [X_t \in A\mid \clZ_t ] = \Prob [X_t^i \in A\mid \clZ_t ], \quad \forall\;A\subset \mathbb{R}^d,\;\;t>0.
\]
In a numerical implementation, a finite number, $N$, of particles is simulated and 
$\Prob [X_t^i \in A\mid \clZ_t ] \approx \frac{1}{N}\sum_{i=1}^N \ind
[ X^i_t\in A]$ by the Law of Large Numbers (LLN).

\smallskip

The challenging part in the numerical implementation of the FPF
algorithm is the solution of the PDE~\eqref{eqn:EL_phi_intro}.  This has
been the subject of a number of recent studies:  In our original FPF
papers, a Galerkin numerical method was proposed; cf.,~\cite{yang2016,taoyang_TAC12}.  A
special case of the Galerkin solution is the constant gain
approximation formula which is often a popular choice in
practice~\cite{yang2016,stano2014,tilton2013,berntorp2015}.  The main issue with the Galerkin approximation
is to choose the basis functions. A proper orthogonal decomposition (POD)-based
procedure to select basis functions is introduced
in~\cite{berntorp2016} and certain continuation schemes appear
in~\cite{matsuura2016suboptimal}.  Apart from the Galerkin procedure,
probabilistic approaches based on dynamic programming appear
in~\cite{Sean_CDC2016}.

\smallskip

In a recent work, we introduced a basis-free {\em kernel-based}
algorithm for approximating the solution of the gain function~\cite{amirCDC2016}.  The
key step is to construct a Markov matrix on the $N$-node graph defined
by the $N$ particles $\{X^i_t\}_{i=1}^N$.
The value of the function $\phi$ for the particles, $\phi(X^i_t)$, is then
approximated by solving a fixed-point problem involving
the Markov matrix.  The
fixed-point problem is shown to be a contraction and the method of
successive approximation applies to numerically obtain the solution.

\smallskip

The present paper presents a continuation and refinement of the
analysis for the kernel-based method.  The contributions are as
follows: A novel formula for the gain function is derived for the
kernel-based approximation.  A procedure
for carrying out error analysis of the approximation is introduced.
Certain asymptotic estimates for bias and variance are derived
for the general nonlinear non-Gaussian case.  Comparison with the
constant gain approximation formula are provided.  These results are 
illustrated with the aid of some numerical experiments.


The outline of the remainder of this paper is as follows: The
mathematical problem of the gain function approximation together with
a summary of known results on this topic appears in~\Sec{sec:prelim}.  
The kernel-based algorithm including the novel formula for gain
function, referred to as (G2), appears in \Sec{sec:kernel}.  
The main theoretical results of this paper including the bias and variance
estimates appear in \Sec{sec:analysis}.  Some numerical experiments for
the same appear in \Sec{sec:numerics}.  

\medskip

\noindent \textbf{Notation.} $\mathbb{Z}_+$ denotes the set of
positive integers and $\mathbb{Z}_+^d$ is the set of $d$-tuples.  
For vectors $x,y\in\Re^d$, the dot product is denoted as $x\cdot y$ and
$|x|:=\sqrt{x\cdot x}$.  
Throughout the paper, it is assumed that the probability measures
admit a smooth Lebesgue density.  A density for a Gaussian random
variable with mean $\mu$ and variance $\Sigma$ is denoted as ${\cal N}(\mu,\Sigma)$.  $C^k$ is used to denote the space of $k$-times
continuously differentiable functions.  For a function $f$, $\nabla f = \frac{\partial f}{\partial x_i}$ is
used to denote the gradient.  
$L^2(\Re^d,\pr)$ is the Hilbert space of square integrable functions on
$\Re^d$ equipped with the inner-product, $\big
<\phi,\psi\big>_{L^2}:=\int\phi(x)\psi(x) \pr(x)\ud x$. 
The associated norm is
denoted as $\|\phi\|^2_2:=\big<\phi,\phi\big>$. The space  
$H^1(\Re^d,\pr)$ is the
space of square integrable functions $\phi$ whose derivative (defined in the
weak sense) is in $L^2(\Re^d,\pr)$.  For the remainder of this paper,
$L^2$ and $H^1$ is used to denote $L^2(\Re^d,\pr)$ and
$H^1(\Re^d,\pr)$, respectively.

\section{Preliminaries}
\label{sec:prelim}
\subsection{Problem Statement}

The Poisson equation~\eqref{eqn:EL_phi_intro} is expressed as,
\begin{flalign}
\label{eqn:EL_phi_prelim}
\text{PDE} && &
\begin{aligned}
-\Delta_\rho \phi & = h-\hah,\\
\int \phi \rho \ud x & = 0 \qquad \text{(zero-mean)},
\end{aligned} &
\end{flalign}
where $\rho$ is a probability density on $\Re^d$, $\Delta_\rho
\phi:=\frac{1}{\rho}\nabla \cdot (\rho\nabla \phi)$. 
The gain function $\k(x):=\nabla\phi(x)$. 

\medskip

\noindent\textbf{Problem statement:} Given $N$ independent
samples $\{X^1,\hdots,X^i,\hdots,X^N\}$ drawn from $\rho$, approximate the gain function $\{\k(X^1),\hdots,\k(X^i),\hdots,\k(X^N)\}$.  The density $\rho$ is not explicitly
known.  

\medskip

The appropriate function space for the solutions
of~\eqref{eqn:EL_phi_prelim} is the co-dimension 1 subspace $L^2_0
:=\{\phi \in L^2; \int \phi \rho \ud x =0\}$ and $H^1_0 :=\{\phi \in
H^1; \int \phi \rho \ud x =0\}$; cf.~\cite{laugesen15,yang2016}.

\subsection{Existence-Uniqueness}

On multiplying both side of~\eqref{eqn:EL_phi_prelim} by test function $\psi$, one obtains the weak-form of the PDE:
\begin{equation}
\int \nabla \phi \cdot\nabla \psi \;\rho \ud x = \int  (h-\hah) \psi \; \rho \ud x,\quad \forall \psi \in H^1.
\label{eq:weakform}
\end{equation}

The following is assumed throughout the paper:
\begin{romannum}
\item {\bf Assumption A1:} The probability density is of the form
  $\rho(x)= e^{-V(x)}$ where $V \in C^2$ with \[\liminf_{x\to\infty}
  \;\; [-\Delta V(x) + \frac{1}{2}|\nabla V(x)|^2]=\infty.\]
\item {\bf Assumption A2:} The function $h,\nabla h \in L^2$. 
\end{romannum}
 
Under the Assumption A1, the density $\rho$ admits a spectral gap (or Poincar\'e inequality) (\cite{bakry2013} Thm 4.6.3), i.e., $\exists \lambda_1 >0$ such that,
\begin{equation*}
\label{eq:poincare}
\int f^2 \, \rho \ud x \leq  \frac{1}{\lambda_1} \int |\nabla f|^2 \, \rho \ud x ,\quad \forall f \in H^1_0.
\end{equation*}
The Poincar\'e inequality implies the existence and uniqueness of a
weak solution to the weighted Poisson equation.  

\medskip

\begin{theorem}{[Theorem~2.2 in~\cite{laugesen15}].}
Assume (A1)-(A2).  Then there exists a unique weak solution
$\phi \in H_0^1(\R^d;\rho)$ satisfying
\eqref{eq:weakform}. Moreover, the gain function $\k=\nabla\phi$ is controlled by the size of the data:
\begin{equation*}
\int |\k|^2 \;\rho \ud x  \le \frac{1}{\lambda_1} \int |h-\hah|^2 \;\rho \ud x.
\end{equation*}
\qed
\end{theorem}

There are two special cases where the exact solution can be found:
\begin{romannum}
\item Scalar case where the state dimension $d=1$;
\item Gaussian case where the density $\rho$ is a Gaussian.
\end{romannum}
The results for these two special cases appear in the following two
subsections. 

\subsection{Exact Solution in the Scalar Case}

In the scalar case (where $d=1$), the Poisson equation is:
\begin{equation*}
-\frac{1}{\rho(x)}\frac{\ud }{\ud x}(\rho(x) \frac{\ud \phi}{\ud
  x}(x)) = h-\hah.
\end{equation*}
Integrating twice yields the solution explicitly,
\begin{equation}
\begin{aligned}
\k(x) = \frac{\ud \phi}{\ud x}(x) &= -\frac{1}{\rho(x)}\int_{-\infty}^x
\rho(z)(h(z)-\hah)\ud z.
\end{aligned}
\label{eq:scalar}
\end{equation}

For the particular choice of $\rho$ as the sum of two Gaussians
${\cal N}(-1,\sigma^2)$ and ${\cal N}(+1,\sigma^2)$ with $\sigma^2=0.2$ and $h(x)=x$, the solution
obtained using~\eqref{eq:scalar} is depicted in \Fig{fig:truesol}.

\begin{figure}
\centering
\includegraphics[width=\columnwidth]{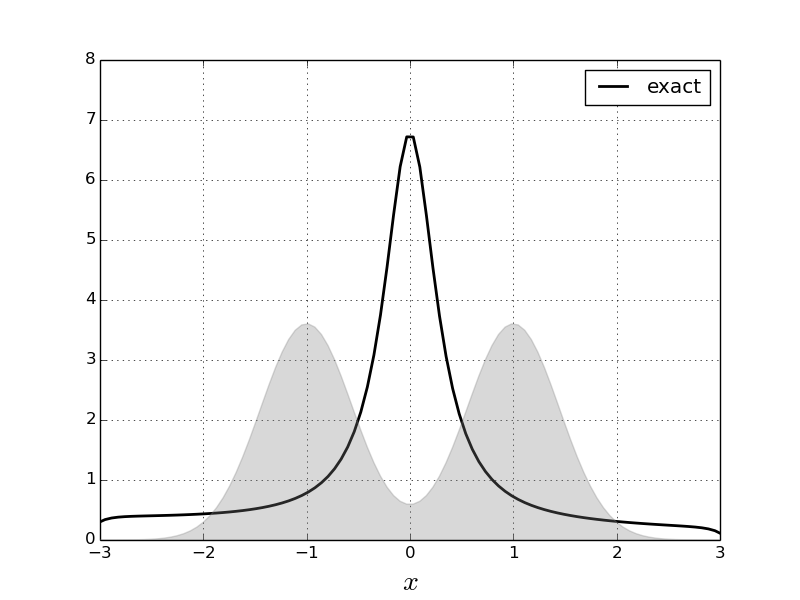}
\vspace*{-0.2in}
\caption{The exact solution to the Poisson equation using the
  formula~\eqref{eq:scalar}.  The density $\rho$ is the sum of two Gaussians
$N(-1,\sigma^2)$ and $N(+1,\sigma^2)$, and $h(x)=x$.  The density is depicted as the shaded curve in the background.}
\vspace*{-0.1in}
\label{fig:truesol}
\end{figure}

\subsection{Exact Spectral Solution for the Gaussian Density} 

Under Assumption (A1), the spectrum is known to be discrete with an ordered sequence of
eigenvalues $0=\lambda_0<\lambda_1\le\lambda_2\le\hdots$ and
associated eigenfunctions $\{e_n\}$ that form a complete orthonormal
basis of $L^2$ [Corollary~4.10.9 in ~\cite{bakry2013}].  The
trivial eigenvalue $\lambda_0=0$ with associated eigenfunction $e_0=1$.
On the subspace of zero-mean functions, the spectral decomposition
yields: For $\phi\in L^2_0$,
\begin{equation*}
\label{eq:spectral_rep}
-\Delta_\rho \phi = \sum_{m=1}^\infty \lambda_m <e_m,\phi>e_m.
\end{equation*}
The spectral gap condition~\eqref{eq:poincare} implies that
$\lambda_1>0$.

The spectral representation~\eqref{eq:spectral_rep} yields the
following closed-form solution of the PDE~\eqref{eqn:EL_phi_prelim}:   
\begin{equation*}
\phi = \sum_{m=1}^N \frac{1}{\lambda_m}<e_m,h-\hah>e_m.
\label{eq:spectralSolution}
\end{equation*}

The spectral representation formula~\eqref{eq:spectralSolution} is
used to obtain the exact solution for the Gaussian case where the eigenvalues and the
eigenfunctions are explicitly known in terms of Hermite polynomials. 

\medskip

\begin{definition}
The {\em Hermite polynomials} are recursively defined as
\begin{equation*}
\hslash_{n+1}(x) = 2 \, \hslash_{n}(x) - \hslash'_n(x),\quad \hslash_0(x) = 1,
\end{equation*} 
where the prime $'$ denotes the derivative.
\qed
\end{definition} 

\medskip

\begin{proposition} Suppose the density $\rho$ is
  Gaussian ${\cal }(\mu,\Sigma)$ where the mean $\mu\in \Re^d$ and the
  covariance $\Sigma$ is assumed to be a
  strictly positive definite symmetric matrix.  Express 
  $\Sigma = VDV^T$ where $D=\text{diag}(\sigma_1^2,\ldots,\sigma_d^2)$
  and $V = [V_1|\hdots|V_j|\hdots|V_d]$ is an orthonormal matrix with
  the $j^{\text{th}}$ column denoted as $V_j \in \Re^d$.  For
  $n=(n_1,\ldots,n_d) \in \mathbb{Z}_+^d$, 
\begin{romannum}
\item The eigenvalues are,
\begin{equation*}
\lambda_n = \sum_{j=1}^d\frac{n_j}{\sigma_j^2}.
\end{equation*} 
\item The corresponding eigenfunctions are,
\begin{equation*}
e_n (x) = \prod_{j=1}^d \hslash_{n_j}(\frac{V_j \cdot (x-\mu)}{\sigma_j}),
\end{equation*}
where $\hslash_{n_j}$ is the Hermite polynomial.
\end{romannum}
\end{proposition}
\medskip

An immediate corollary is that the first non-zero eigenvalue is
$\frac{1}{\sigma_{\text{max}}^2}$ and the corresponding eigenfunction
is $\psi(x)=V_{\text{max}} \cdot (x-\mu)$, where
$\sigma_{\text{max}}^2$ is the largest eigenvalue of the covariance
matrix and $V_{\text{max}}\in\Re^d$ is the corresponding eigenvector. 

\medskip

\begin{example} Suppose the density $\rho$ is a
  Gaussian ${\cal }(\mu,\Sigma)$.
\begin{romannum}
\item The observation function $h(x) =
  H \cdot x$, where $H \in \Re^d$.  Then, $\phi = \Sigma H
  \cdot x$ and the gain function $\k = \Sigma \, H$ is the Kalman
  gain.
 
\item Suppose $d=2$, $\mu = [0,0]$, $\Sigma=\begin{bmatrix} \sigma_1^2
    & 0 \\0&\sigma_2^2\end{bmatrix}$, and the observation function
  $h(x_1,x_2)=x_1\,x_2$.  Then,
\[
\phi(x) =
  \frac{1}{\frac{1}{\sigma_1^2}+\frac{1}{\sigma_2^2}}x_1x_2
\quad\text{and}\quad
\k(x_1,x_2) =
\frac{1}{\frac{1}{\sigma_1^2}+\frac{1}{\sigma_2^2}} \begin{bmatrix}
  x_2 \\ x_1\end{bmatrix}.
\]

\end{romannum}
\qed
\end{example}

\medskip

In the general non-Gaussian case, the solution is not known in an
explicit form and must be numerically approximated.  Note that even in
the two exact cases, one may need to numerically approximate the solution
because the density is not given in an explicit form.  A popular
choice is the constant gain approximation briefly described next.     

\begin{figure}[t]
\centering
\includegraphics[width=1.0\columnwidth]{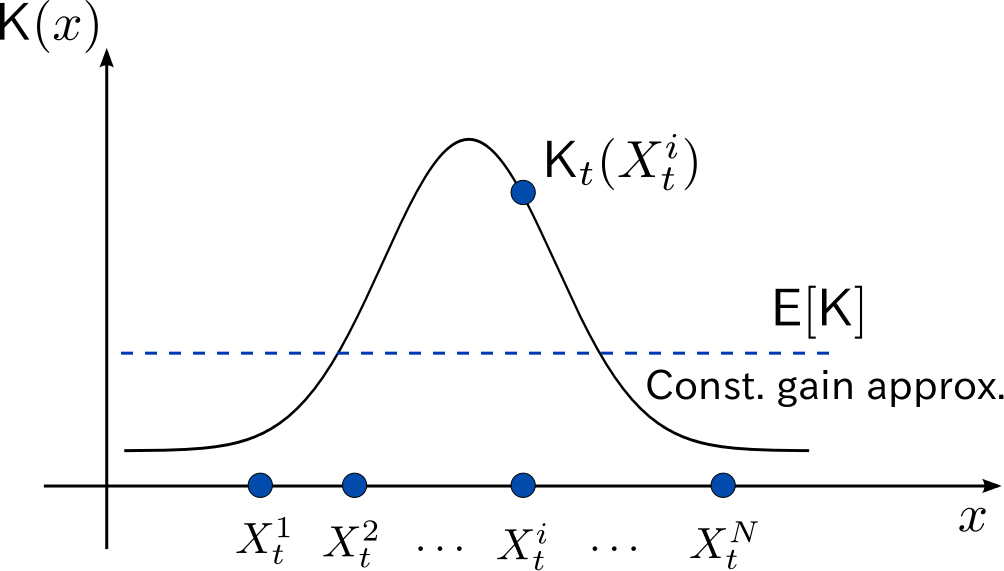}
\caption{Constant gain approximation:  Approximating nonlinear $\k$ by
its expected value ${\sf E}[\k]$.}
\vspace*{-0.2in}
\label{fig:const-gain-approx}
\end{figure} 

\subsection{Constant gain approximation}

The constant gain approximation is the best -- in the least-square sense -- constant approximation of the
gain function (see~\Fig{fig:const-gain-approx}).  Precisely, consider the following least-square
optimization problem:
\begin{equation*}
\kappa^\ast = \arg \min_{\kappa \in \Re^d} {\sf E}_\rho[|\k - \kappa|^2].
\end{equation*}
By using a standard sum of square argument, 
$
\kappa^\ast = {\sf E}_\rho[\k].
$
The expected value admits an explicit formula: In the
weak-form~\eqref{eq:weakform}, choose the test functions to be the
coordinate functions: $\psi_k(x) = x_k$ for
$k=1,2,\hdots,d$. Writing $\psi(x) = (\psi_1,\psi_2,\hdots,\psi_d)^T = x$,
\begin{align*}
\kappa^* = \Expect_\rho[\k] &= \Expect_\rho [  (h-\hah) \psi ]  = \int_{\Re^d} (h(x)-\hah)\;
x \; \rho(x) \ud x.
\end{align*}
On computing the integral using only the particles, one obtains the
formula for the gain function approximation:
\[
\k \approx \frac{1}{N}\sum_{i=1}^N\; (h(X^i)-\hat{h}^{(N)}) \; X^i,
\]
where $\hat{h}^{(N)} = N^{-1} \sum_{i=1}^N h(X^i)$. 
This formula is referred to as the {\em constant gain approximation}
of the gain function;
cf.,~\cite{yang2016}.  It is a popular choice in applications~\cite{yang2016,stano2014,tilton2013,berntorp2015}..

\section{Kernel-based Approximation} 
\label{sec:kernel}

\noindent \textbf{Semigroup:} The spectral gap condition~\eqref{eq:poincare} implies that
$\lambda_1>0$.  Consequently, the semigroup
\begin{equation}
e^{t\Delta_\rho}\phi := \sum_{m=1}^{\infty} e^{-t\lambda_m}<e_m,\phi>e_m
\label{eq:semigroup}
\end{equation} 
is a strict contraction on the subspace $L^2_0$.  It is also easy to see that
$\mu$ is an invariant measure and $\int e^{t\Delta_\rho}\phi(x)
\ud\mu(x) = \int \phi(x) \ud\mu(x) = 0$ for all $\phi\in L^2_0$. 

\medskip

The semigroup
formula~\eqref{eq:semigroup} is used to obtain the solution of the
Poisson equation~\eqref{eqn:EL_phi_prelim} by solving the following fixed-point equation for
{\em any} fixed positive value of $t$:
\begin{equation}
\phi = e^{t\Delta_\rho} \phi+ \int_0^t e^{s\Delta_\rho} (h-\hat{h}) \ud s.
\label{eq:fixed1}
\end{equation}
A unique solution exists because $e^{t\Delta_\rho}$ is a contraction
on $L^2_0$.  

\medskip

\noindent \textbf{Kernel-based method:}
In the kernel-based algorithm, one approximates the solution of the
fixed point problem~\eqref{eq:fixed1} by approximating the semigroup
by an integral operator for $t=\epsilon$.  The approximation,
introduced in~\cite{amirCDC2016}, has three main steps:
\begin{align}
\text{Exact}:&~&\phi &= e^{\epsilon\Delta_\rho} \phi+ \int_0^\epsilon
e^{s\Delta_\rho} (h-\hah) \ud s \label{eq:fixedPt}\\
\text{Kernel approx:}&~ &\phieps &= \Teps \phieps+ \int_0^\epsilon T_s
(h-\hah)  \ud s \label{eq:fixed-point-eps}\\
\text{Empirical approx:}&~ &\phiepsN &= \TepsN \phiepsN +
\int_0^{\epsilon}T_s^{(N)} (h-\hah)  \ud s
\label{eq:fixed-point-epsN}
\end{align}
The justification for these steps is as follows:
\begin{romannum}
\item A solution of the Poisson equation~\eqref{eqn:EL_phi_prelim} is also a solution of the
  fixed-point problem~\eqref{eq:fixedPt}  where $\epsilon>0$ is
  arbitrary. A unique solution exists because $e^{\epsilon\Delta_\rho} $ is
  contraction on $L^2_0$.
\medskip
\item The Kernel approximation~\eqref{eq:fixed-point-eps} involves approximating
  the semigroup $e^{\epsilon\Delta_\rho}$ by an integral operator
  $\Teps$,
\begin{equation}
\Teps f := \int_{\Re^d}k_\epsilon(x,y)f(y)\rho(y)\ud y
\label{eq:Teps}
\end{equation} 
where the exact form of $\keps$ appears in the Appendix, where it is
also shown that $e^{\epsilon\Delta_\rho}  \approx \Teps$ as $\epsilon
\downarrow 0$.  The approximation of the semigroup by the integral
operator appears in~\cite{coifman,hein2006}.  

\medskip
\item The empirical approximation~\eqref{eq:fixed-point-epsN} involves approximating the integral
  operator empirically in terms of the particles,
\begin{equation}
\TepsN f(x) := \frac{1}{N}\sum_{i=1}^N \kepsN(x,X^i)f(X^i),
\label{eq:TepsN}
\end{equation} 
justified by the LLN.  

\end{romannum}

\medskip

The gain $\vepsN$ is computed by taking the gradient of the
fixed-point equation
\eqref{eq:fixed-point-epsN}.  For this purpose, denote,
\begin{equation}
\begin{aligned}
\nabla \TepsN f (x) &:= \frac{1}{N}\sum_{j=1}^N \nabla\kepsN(x,X^j)f(X^j) \\
&= \frac{1}{2\epsilon} \bigg[\frac{1}{N}\sum_{i=1}^N \kepsN(x,X^i)X^if(X^i) \\- &\left(\frac{1}{N}\sum_{i=1}^N \kepsN(x,X^i)X^i\right)\left(\frac{1}{N}\sum_{i=1}^N \kepsN(x,X^i)f(X^i)\right)\bigg].
\end{aligned}
\end{equation}
Next, two approximate formulae for $\vepsN$ are presented based on two different approximations of the integral $\int_0^\epsilon
T_s^{(N)}(h-\hat{h}) \ud s$:

\newP{Approximation 1} The integral term is approximated by $\epsilon
(h-\hat{h})$ and the resulting formula for the gain is,
\begin{flalign}
&\text{(G1)}&\quad \KepsN(x)&:= \nabla \TepsN \phiepsN(x)  + \epsilon\nabla h(x).&
\end{flalign}

\medskip

By approximating the integral term differently, one can avoid the
need to take a derivative of $h$.

\medskip 

\newP{Approximation 2} The integral term is approximated by $\TepsN
(h-\hat{h})$.  The resulting formula for the gain is,
\begin{flalign}
&(\text{G2})&\quad \KepsN(x)&:=\nabla \TepsN \phiepsN(x)  + \epsilon\nabla \TepsN (h-\hat{h})(x).&
\label{eq:kepsN-approx}
\end{flalign}

\medskip

\begin{remark} Although $\nabla\TepsN$ and $\KepsN$ are
ultimately important in the numerical algorithm (described next), it
is useful to introduce the limiting (as $N\rightarrow\infty$)
variables $\nabla\Teps$ and $\Keps$.  The operator $\nabla\Teps$ is
defined as follows:
\begin{equation}
\begin{aligned}
\nabla \Teps f(x) &= \int_{\Re^d} \nabla_x\keps(x,y) f(y)\pr(y)\ud y \\
&= \frac{1}{2\epsilon}\bigg[\Teps(ef)-\Teps(e)\Teps(f)\bigg](x)
\end{aligned}
\end{equation}
where $e$ is the identity function $e(x)=x$.

In terms of $\nabla\Teps$, the gain function $\Keps$ is defined by
taking of the gradient of the fixed-point
equation~\eqref{eq:fixed-point-eps}.  This leads to the limiting
counterpart of the approximation (G1) and (G2).  In particular,
analogous to (G2),
\[
\Keps(x):=\nabla \Teps \phieps(x)  + \epsilon\nabla \Teps (h-\hat{h})(x),
\]
where $\phieps$ is the solution of the fixed-point equation~\eqref{eq:fixed-point-eps}.
\end{remark}

\medskip

\noindent \textbf{Numerical Algorithm:} A numerical implementation
involves the following steps:
\begin{romannum}
\item Assemble a $N\times N$ Markov matrix to approximate the finite rank operator
  $\TepsN$ in~\eqref{eq:TepsN}.  The $(i,j)$-entry of the matrix is
  given by,
\[
\Tm_{ij}= \frac{1}{N}\sum_{j=1}^N \kepsN(X^i,X^j).
\]

\item Use the method of successive approximation to solve the
  discrete counterpart of the fixed-point equation~\eqref{eq:fixed-point-epsN},
\begin{equation}
\Phi = \Tm \Phi + \epsilon (\hm-\hat{h}^{(N)})
\label{eq:fixed-point-matrix}
\end{equation} 
where $\Phi:=[\phiepsN(X^1),\ldots,\phiepsN (X^N)] \in \Re^N$ is the
(unknown) solution, $\hm:=[h(X^1),\ldots,h(X^N)]
\in \Re^N$ is given, and $\hat{h}^{(N)} = \frac{1}{N}\sum_{i=1}^N
h(X^i)$.  In filtering applications, the solution from the previous
time-step is typically used to initialize the algorithm.


\item Once $\Phi$ has been computed, the gain function
  $\{\k(X^1),\hdots,\k(X^i),\hdots,\k(X^N)\}$ is obtained by using
  either (G1) or (G2).  Note that the discrete counterpart of $\nabla
  \TepsN$ is obtained using the Markov matrix $\Tm$.  
\end{romannum}

\medskip

The overall algorithm is tabulated as Algorithm 1 where (G2) is
used for the gain function approximation. 
\begin{algorithm}
\caption{Kernel-based gain function approximation}
\begin{algorithmic}
\REQUIRE $\{X^i\}_{i=1}^N$, $H:=\{h(X^i)\}_{i=1}^N$,$\Phi_0:=\{\phi_0(X^i)\}_{i=1}^N$ 
\ENSURE $\Phi:=\{\phi(X^i)\}_{i=1}^N$, $\{\nabla \phi(X^i)\}_{i=1}^N$ \medskip
\STATE Calculate $g_{ij}:=\exp(-|X^i-X^j|^2/4\epsilon)$ for $i,j=1$ to $N$.\medskip
\STATE Calculate $k_{ij}:=\frac{g_{ij}}{\sqrt{\sum_l g_{il}}\sqrt{\sum_l g_{jl}}}$ for $i,j=1$ to $N$.
\STATE Calculate $T_{ij}:=\frac{k_{ij}}{\sum_l k_{il}}$ for $i,j=1$ to $N$.
\STATE Calculate $\hat{h}^{(N)}=\frac{1}{N}\sum_{i=1}^N H_i$. \medskip
\FOR {$t=1$ to  $T$}
\STATE Solve $\Phi_{t}= T \Phi_{t-1} + \epsilon (H-\hat{h})$. \medskip
\STATE $\Phi_{t} = \Phi_t - \frac{1}{N}\sum_{i=1}^N \Phi_{t,i}$
\ENDFOR
\STATE Calculate 
\[
\k (X^i)= \frac{1}{2\epsilon}\sum_{j=1}^N \left[T_{ij}(\Phi_j+\epsilon(H_j - \hat{h}))\left(X^j- \sum_{k=1}^N T_{ik}X^k\right)\right]
\]
\end{algorithmic}
\label{alg:kernel}
\end{algorithm}

\section{Error Analysis}
\label{sec:analysis}

The objective is to characterize the approximation error
$\expect[\|\k_\epsilon^{(N)} - \k\|_2]$.  Using the triangle inequality,
\begin{equation}
\expect[\|\k_\epsilon^{(N)} - \k\|_2] \;\; \leq \;\;
\underbrace{\expect[\|\k_\epsilon^{(N)} -
  \k_\epsilon\|_2]}_{\text{Variance}} \; +\; \underbrace{\|\k_\epsilon - \k\|_2}_{\text{Bias}},
\end{equation}
where $\k=\nabla\phi$ denotes the exact gain function, and 
$\k_\epsilon(x)=\nabla\phi_{\epsilon}(x)$ and
$\k_\epsilon^{(N)}(x)=\nabla\phi_{\epsilon}^{(N)}(x)$ are defined by
taking the gradient of the fixed-point
equation~\eqref{eq:fixed-point-eps} and~\eqref{eq:fixed-point-epsN},
respectively.  


The following Theorem provides error estimates for the gain function
in the asymptotic limit as $\epsilon \downarrow
0$ and $N \to \infty$.
These estimates apply to either of the two approximations, (G1) or
(G2), used to obtain the gain function.  A sketch of the proof appears
in the Appendix.  

\medskip

\begin{theorem} 
Suppose the assumptions (A1)-(A2) hold for the density $\rho$ and the
function $h$, with spectral gap constant $\lambda_1$.  Then 
\begin{enumerate}
\item (Bias) In the asymptotic limit as $\epsilon \downarrow 0$, 
\begin{equation}
\|\k_\epsilon - \k\|_2 \; \le \; C\epsilon  + \text{h.o.t.},
\end{equation}
\item(Variance) In the asymptotic limit as $\epsilon \downarrow 0$ and $N \to \infty$,
\begin{equation}
\begin{aligned}
\expect[\|\k_\epsilon - \k_\epsilon^{(N)}\|_2]&\; \leq \;
\frac{C}{N^{1/2}\epsilon^{1+d/4}} + \text{h.o.t.},
\end{aligned}
\end{equation}
where the constant $C$ depends upon the function $h$.
\end{enumerate}
\label{thm:error-bound}
\end{theorem}
\begin{figure}[t]
\centering
\includegraphics[width=0.8\columnwidth]{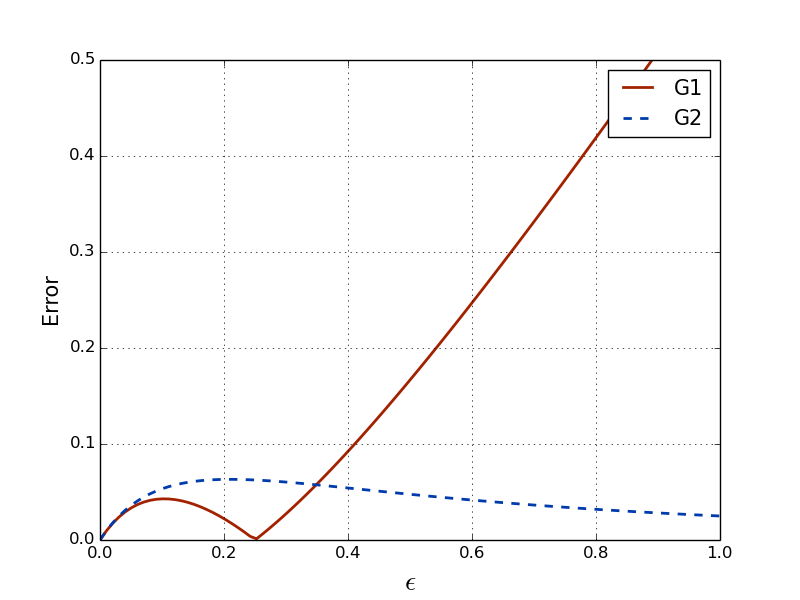}
\caption{The figure shows explicit error $\|\k_\epsilon - \k\|_2$}
\label{fig:eps-error-explicit}
\end{figure} 
\subsection{Difference between (G1) and (G2)}

In the asymptotic limit as $\epsilon\downarrow 0$, the two
approximations (G1) and (G2) yield identical error estimates.  The difference arises
as $\epsilon$ becomes larger.  The following Proposition provides
explicit error estimates for the bias in the special linear Gaussian case.

\medskip

\begin{proposition}
Suppose the density $\rho$ is a Gaussian $N(0,\sigma^2I)$ and
$h(x)=H\cdot x$.  Then the bias for the two approximations is given by
the following closed-form formula: 
\begin{align}
\text{Bias for (G1):} \quad \|\k_\epsilon - \k\|_2 &= \epsilon \frac{\sigma^2 - 4\epsilon}{\sigma^2 + 4\epsilon} |H|, \\
\text{Bias for (G2):} \quad \|\k_\epsilon - \k\|_2 &= \epsilon \frac{\sigma^6}{(\sigma^2+4\epsilon)(\sigma^4+3\epsilon\sigma^2 + 4\epsilon^2)} |H|. 
\end{align}
\label{prop:explicit-error}
\end{proposition}

\begin{figure*}[t]
\centering
\begin{tabular}{cc}
\subfigure[]{
\includegraphics[width=1.0\columnwidth]{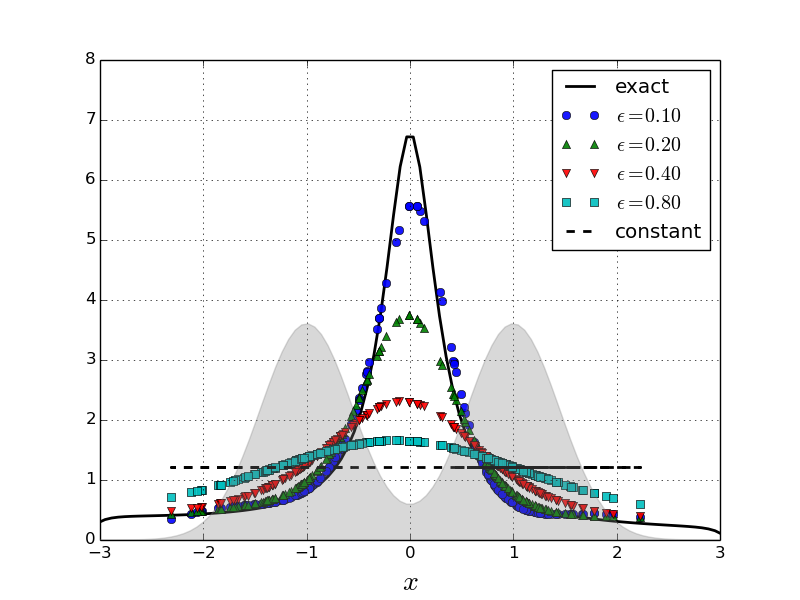}
} \label{fig:gain-bomodal-lin-1}&
\subfigure[]{
\includegraphics[width=1.0\columnwidth]{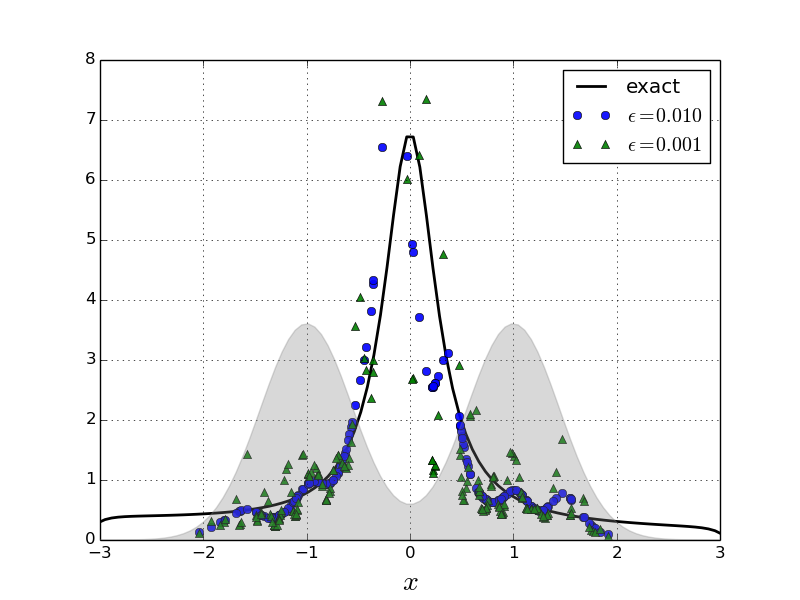}
} 
\end{tabular}
\caption{Numerical gain function approximation (G2) for a range of
  $\epsilon$.  The dimension $d=1$ and the number of particles
  $N=200$.  }
\label{fig:gain-bimodal-lin}
\end{figure*}

Note that the bias has the same scaling, $\sim \epsilon |H|$, as
$\epsilon\downarrow 0$.  However as $\epsilon$ gets larger, the two
approximations behave very differently.  For (G1), the bias grows
unbounded as $\epsilon\rightarrow \infty$.  Remarkably, for (G2), the
bias goes to zero as $\epsilon\rightarrow \infty$.
Figure~\ref{fig:eps-error-explicit} depicts the bias error for a
scalar example where $\sigma^2=1$ and $H=1$.  


The following Proposition shows that the limit
$\epsilon\rightarrow\infty$ is well-behaved for the (G2) approximation
more generally.  In fact, one recovers the constant gain
approximation in that limit.

\medskip

\begin{proposition}
Consider the gain approximation (G2) given
by~\eqref{eq:kepsN-approx}. Then,
\begin{align*}
\lim_{\epsilon\rightarrow\infty} \;\;\k_{\epsilon} &= {\sf E}[K],\\
\lim_{\epsilon\rightarrow\infty} \;\;\k_{\epsilon}^{(N)} &=  \frac{1}{N}\sum_{i=1}^N(h(X^i)-\hat{h}^{(N)})X^i.
\end{align*}
\qed
\label{prop:eps-infty}
\end{proposition}

\section{Numerics}\label{sec:numerics}

Suppose the density $\pr$ is a mixture of two Gaussians,
$\frac{1}{2}{\cal N}(-\mu,\sigma^2I) + \frac{1}{2}{\cal
  N}(+\mu,\sigma^2I)$, where $\mu=[1,0,\ldots,0]\in \Re^d$, and
$\sigma^2=0.2$. The observation function $h(x)=x_1$.  In this case,
the exact gain function $\k(x) = [\k_1(x),0,\ldots,0]$ where
$\k_1(\cdot)$ is obtained using the explicit formula~\eqref{eq:scalar}
as in the scalar case.

Figure~\ref{fig:gain-bimodal-lin} depicts a comparison between the
exact solution and the approximate solution obtained using the kernel
approximation formula (G2).  The dimension $d=1$ and the number of
particles $N=200$.  
\begin{itemize}
\item The part~(a) of the
figure depicts the gain function for a range of (relatively large)
$\epsilon$ values $\{0.1,0.2,0.4,0.8\}$ where the error is dominated
by the bias.  The constant gain approximation is also depicted and,
consistent with Proposition~\ref{prop:eps-infty}, the (G2)
approximation converges to the constant as $\epsilon$ gets larger.  

\medskip

\item The part~(b) of the figure depicts a comparison for a range of
(very small) $\epsilon$ values $\{0.01,0.001\}$.  At $N=200$
particles, the error in this range is dominated by the variance.  This
is manifested in a somewhat irregular spread of the particles for
these $\epsilon$ values.    
\end{itemize}

In the next study, we experimentally evaluated the error for a range
of $\epsilon$ and $d$, again with a fixed $N=200$.  For a single
simulation, the error is defined as
\[
\text{Error}:=\sqrt{\frac{1}{N}\sum_{i=1}^N |\k_\epsilon^{(N)}(X^i)-\k(X^i)|^2}.
\]
Figure~\ref{fig:bimodal-lin-error}
and~\ref{fig:bimodal-lin-error-loglog} depict the averaged error
obtained from averaging over $M=100$ simulations.  In each simulation,
the parameters $\epsilon$ and $d$ are fixed but a different
realization of $N=200$ particles is sampled from the density $\pr$.  

\medskip

\begin{itemize}
\item Figure \ref{fig:bimodal-lin-error} depicts the averaged error as
  $\epsilon$ and $d$ are varied. As $\epsilon$ becomes large, the
  kernel gain converges the constant gain formula. For relatively
  large values of $\epsilon$, the error is dominated by bias which is
  insensitive to the size of dimension $d$.  

\medskip

\item Figure \ref{fig:bimodal-lin-error-loglog} depicts the averaged
  error for small values of $\epsilon$.  The logarithmic scale is used
  to better assess the asymptotic characteristics of the error as a function of
  $\epsilon$ and $d$.  Recall that the estimates in
  Theorem~\ref{thm:error-bound} predict that the error scales as
  $\epsilon^{-1-d/4}$ for the small $\epsilon$ large $N$ limit.  To verify
  the prediction, an empirical exponent was computed by fitting a
  linear curve to the error data on the logarithmic scale.  The
  empirical exponents together with the error estimates predicted by
  Theorem~\ref{thm:error-bound} are tabulated in Table
  \ref{tab:exponents}.  It is observed that the empirical exponents
  are smaller than the predictions.  The gap suggests that the error
  bound may not be tight.  A more thorough comparison is a subject
  of continuing investigation.   
\end{itemize}
\begin{table}[h]
\centering
\begin{tabular}{c|cccc}
$d$ & $1$ & $2$ & $3$ & $4$ \\ \hline
$1+d/4$ & $1.25 $ & $1.5$ & $1.75$ & $2.0$ \\ \hline 
$\alpha$ & $0.83$ & $1.12$ & $1.36$ & $1.56$ 
\end{tabular}
\caption{Comparison of empirically obtained exponents ($\alpha$) with the
  theoretical exponents $1+d/4$.  Empirical exponents are obtained by
  curve fitting the data in Fig.~\ref{fig:bimodal-lin-error-loglog}.}
\label{tab:exponents} 
\end{table}
\begin{figure*}[t]
\centering
\begin{tabular}{cc}
\subfigure[]{
\includegraphics[width=1.0\columnwidth]{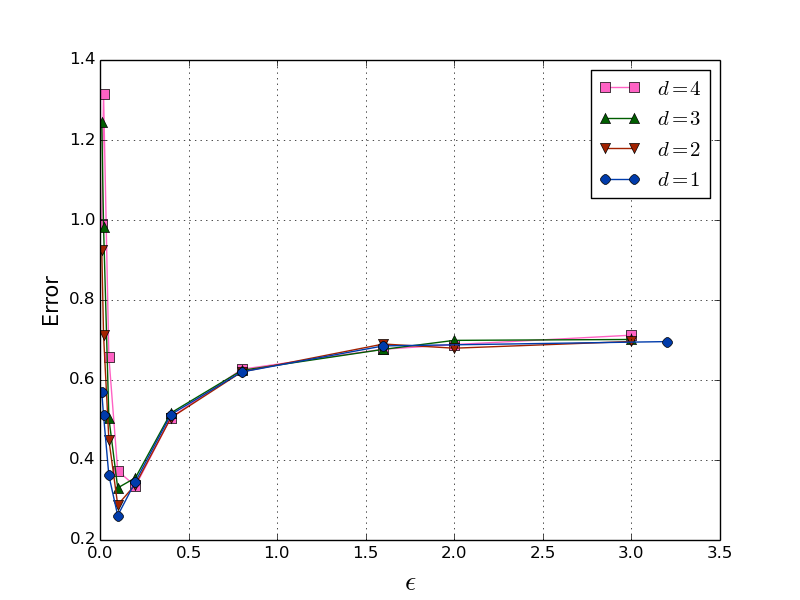}
\label{fig:bimodal-lin-error}
} &
\subfigure[]{
\includegraphics[width=1.0\columnwidth]{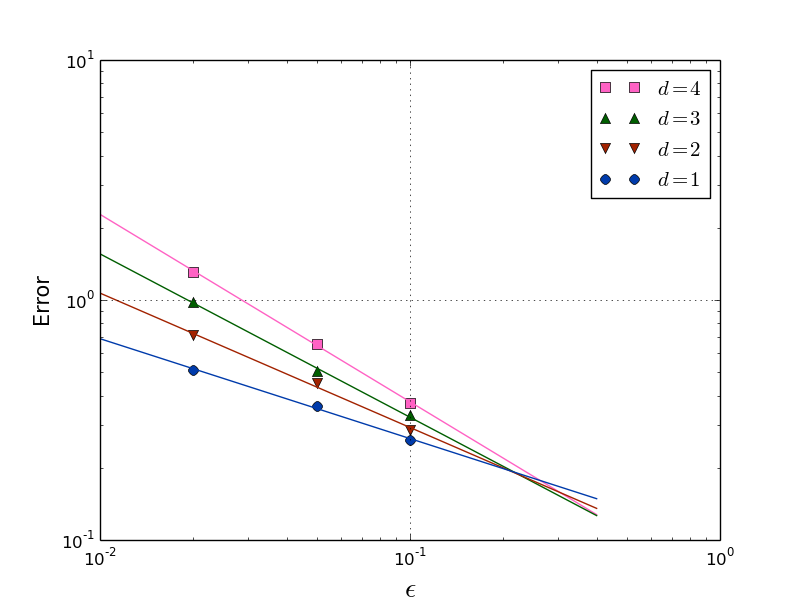}
\label{fig:bimodal-lin-error-loglog}
} 
\end{tabular}
\caption{Averaged error over $M=100$ simulations with  $N=200$ particles: (a) Linear scale
  over a range of $\epsilon$ and (b) Logarithmic scale for small
  $\epsilon$.}
\label{fig:bomodal-lin}
\end{figure*}

\bibliographystyle{plain}
\bibliography{fpfbib,ref}
\appendix 
\subsection{Kernel-based algorithm}
This section provides additional details for the kernel-based
algorithm presented in \Sec{sec:kernel}. We begin with some
definitions and then provide justification for the three main steps,
fixed-point equations~\eqref{eq:fixedPt}-\eqref{eq:fixed-point-epsN}.

\newP{Definitions} The Gaussian kernel is denoted as $g_\epsilon(x,y):=\exp(-\frac{\|x-y\|^2}{4\epsilon})$. 
The approximating family of operators $\{\Teps,\epsilon \ge 0\}$ are
defined as follows: For $f:\Re^d \to \Re$, 
\begin{equation}
\Teps f(x) := \int_{\Re^d}k_\epsilon(x,y)f(y)\pr(y)\ud y,
\label{eq:Teps-appendix}
\end{equation} 
where
\begin{equation*}
\keps:(x,y)= \frac{1}{n_\epsilon(x)}\frac{g_\epsilon(x,y)}{\sqrt{\int
g_\epsilon(y,z) \rho(z) \ud z}}, \label{eq:keps}
\end{equation*}
and $n_\epsilon$ is the normalization factor, chosen such that $\int
\keps(x,y)\rho(y) \ud y=1$. 
The finite-$N$ approximation of these operators, denoted as
$\{\TepsN\}_{\epsilon \geq 0, N \in \mathbb{N}}$, is defined as,
\begin{equation}
\TepsN f(x) := \frac{1}{N}\sum_{i=1}^N \kepsN(x,X^i)f(X^i),
\label{eq:TepsN-appendix}
\end{equation} 
where 
\begin{equation}
\kepsN(x,y) :=  \frac{1}{n_\epsilon^{(N)}(x)}\frac{\geps(x,y) }{\sqrt{\sum_{j=1}^N g_\epsilon(y,X^j)}},
\label{eq:kepsN}
\end{equation}
and $n_\epsilon^{(N)}$ is chosen such that $\frac{1}{N}\sum_{i=1}^N\keps(x,X^i)=1$.



\newP{Justification of the fixed-point equations~\eqref{eq:fixedPt}-\eqref{eq:fixed-point-epsN}} 

\medskip

\noindent (i) Definition of the semigroup $e^{\epsilon\Delta_\pr}$ implies:
\begin{equation*}
e^{\epsilon\Delta_\pr} f = f + \int_0^\epsilon e^{s\Delta_\pr}\Delta_\pr f \ud s
\end{equation*}
On choosing $f= \phi$ where $\Delta_\pr \phi = - (h-\hat{h})$ yields the exact fixed point equation \eqref{eq:fixedPt}.

\medskip

\noindent (ii)  The justification for approximation $\Teps \approx
e^{\epsilon\Delta_\pr}$ is the following Lemma. 

\begin{lemma} Consider the family of Markov operators
  $\{\Teps\}_{\epsilon \geq 0}$. Fix a smooth function $f$. Then  
\begin{equation}
\Teps f(x) = f(x) + \epsilon \Delta_\pr f(x) +O(\epsilon^2).
\label{eq:Teps-lemma}
\end{equation}
\label{lem:Teps}
\end{lemma}  
\begin{proof}
Introduce the heat semigroup $\Geps$ as,
\begin{equation}
G_\epsilon f (x) := \int_{\Re^d} g_\epsilon(x-y) f(y) \pr(y)\ud y.
\label{eq:Geps}
\end{equation}
The following are the two properties of the heat semigroup:
\begin{align}
G_0 f(x) &= \pr(x)f(x), \label{eq:heat-semigp-p1}\\
\frac{\ud}{\ud \epsilon}G_\epsilon f(x) \big|_{\epsilon=0}&= \Delta (\pr f)(x).\label{eq:heat-semigp-p2}
\end{align}
In terms of $\Geps$, the operator $\Teps$ is expressed as,
\begin{equation}
\Teps f(x) = \frac{1}{n_\epsilon(x)}\Geps (\frac{f}{\sqrt{\Geps
    1}})(x),
\label{eq:GepsTeps}
\end{equation}
where $n_\epsilon(x) = \Geps (\frac{1}{\sqrt{\Geps 1}})(x)$. 
The Taylor expansion of $\Teps f(x)$ yields,
\begin{align*}
T_\epsilon f(x) &= T_0 f(x) + \epsilon \frac{\ud }{\ud \epsilon} \Teps f (x)\big|_{\epsilon=0} + O(\epsilon^2).
\end{align*}
Now, the properties~\eqref{eq:heat-semigp-p1} and
\eqref{eq:heat-semigp-p2} can be used to show that $T_0 f(x)=f(x)$ and
$\frac{\ud }{\ud \epsilon} \Teps f (x)\big|_{\epsilon=0} = \Delta_\pr
f(x)$. 
\end{proof}

\medskip

\noindent (iii) The justification for the third step is Law of Large
numbers. Moreover the following lemma provides a bound for the $L^2$ error. 
\begin{lemma} 
Consider the Markov operators $\Teps$ and $\TepsN$ defined in \eqref{eq:Teps-appendix} and \eqref{eq:TepsN-appendix}. 
Then $\forall f \in L^2(\rho)$,
\begin{equation*}
\expect\left[\|\Teps f  -\TepsN f\|_{L^2(\pr)}^2\right] \leq \frac{C}{N\epsilon^{d/2}}.
\end{equation*} 
\label{lem:TepsN}
\end{lemma}  
\begin{proof}
The bound is proved by explicitly evaluating the error
$\expect\left[\|\Geps f  -G_\epsilon^{(N)}f\|_{L^2(\pr)}^2\right]$
where $\Geps$ is defined in \eqref{eq:Geps} and $G_\epsilon^{(N)}
f(x):=\frac{1}{N}\sum_{i=1}^N g_\epsilon(x,X^i)f(X^i)$.  Since
$\{X^i\}$ are i.i.d.,
\[
\Expect\left[\|\Geps f  -G_\epsilon^{(N)}f\|_2\right] \leq \frac{1}{N}\int\int g_\epsilon^2(x,y)f(y)^2 \pr(y)\pr(x)\ud y \ud x.
\]
Subsequently, using the fact that $g_\epsilon^2(x,y) =
\frac{C}{\epsilon^{d/2}}g_{2\epsilon}(x,y)$ and $\int g_{2\epsilon}(x,y)\pr(x)\ud x \leq C$, one obtains,
\[
\expect\left[\|\Geps f  -\GepsN f\|_{L^2(\pr)}^2\right] \leq \frac{C\|f\|^2_{L^2(\pr)}}{N\epsilon^{d/2}},
\]
and the estimate follows because of~\eqref{eq:GepsTeps}.
\end{proof}

\subsection{Sketch of the Proof of Theorem \ref{thm:error-bound}}
\newP{Estimate for Bias} The crucial property is that $\Teps$ is a
bounded strictly contractive operator on $H_0^1$ with
\begin{equation}
\|(I-\Teps)^{-1}\|_{H^1_0(\pr)} = \frac{1}{\epsilon \lambda_1} +
O(1),
\label{eq:Teps_norm}
\end{equation}
where $\lambda_1$ is the spectral bound for $\Delta_{\rho}$.  Since
$\phieps$ solves the fixed-point equation~\eqref{eq:fixed-point-eps},
\[
\phieps = \Teps \phieps + \epsilon (h-\hah) + O(\epsilon^2).
\]
Therefore,
\begin{align*}
\phi - \phieps & = \phi - \Teps \phi  + \Teps (\phi - \phieps) - \epsilon
(h-\hah) + O(\epsilon^2) \\
& = -\epsilon\Delta_{\rho}\phi + \Teps (\phi - \phieps) - \epsilon
(h-\hah) + O(\epsilon^2), 
\end{align*}
where we have used Lemma~\ref{lem:Teps}.  Noting $-\Delta_{\rho}\phi =
(h-\hah)$,
\[
\phi - \phieps = \Teps (\phi - \phieps) + O(\epsilon^2).
\]
The bias estimate now follows from using the norm
estimate~\eqref{eq:Teps_norm}.



\newP{Estimate for variance} The variance estimate follows from using
Lemma~ \ref{lem:TepsN}.  The key steps are to show that 
\[
\|\TepsN f - \Teps f\|_{H^1_0(\rho)} \to 0\quad \text{a.s},
\]
which follows from the LLN, and that $\TepsN$ are bounded and compact
on $H^1_0(\pr)$. This allows one to conclude that
$\|(I-\TepsN)^{-1}\|_{H^1_0(\pr)}$ is bounded.  These forms of
approximation error bounds in a somewhat more general context of
compact operators appears in [Chapter 7 of~\cite{hutson2005}].


\subsection{Proof of proposition \ref{prop:explicit-error}}
For the Gaussian density ${\cal N}(0,\sigma^2 I)$, the completion of square is used to obtain
an explicit form for the operator $\Teps$:
\begin{equation*}
\Teps f(x) = \int \frac{1}{\sqrt{4\pi \epsilon(1-\delta_\epsilon)}}\exp\left[-\frac{(y-(1-\delta_\epsilon)x)^2}{4\epsilon(1-\delta_\epsilon)}\right]f(y)\ud y,
\end{equation*}
where
$\delta_\epsilon:=\epsilon\frac{\sigma^2+4\epsilon}{\sigma^4+3\epsilon\sigma^2+4\epsilon^2}$. For
the linear function $h(x)=H\cdot x$, the fixed-point
equation~\eqref{eq:fixed-point-eps} admits an explicit solution,
\begin{equation*}
\phi_\epsilon = \frac{\epsilon}{\delta_\epsilon}H\cdot x
\end{equation*}
where we used the fact that $\Teps x = 1-\delta_\epsilon x$. 

Since the solution $\phi_\epsilon$  is known in an explicit form, one can easily compute
the gain function solution in an explicit form:
\begin{equation*}
\begin{aligned}
\text{(G1)}\quad\k_\epsilon&=\frac{\epsilon}{\delta_\epsilon}H
=\sigma^2 H-\epsilon\frac{\sigma^2-4\epsilon}{\sigma^2 + 4\epsilon}H,\\
\text{(G2)}\quad\k_\epsilon
&= \sigma^2 H+ \frac{\epsilon\sigma^6}{(\sigma^2 + 4\epsilon)(\sigma^4 + 3\epsilon\sigma^2 + 4\epsilon^2)}H.
\end{aligned}
\end{equation*}
The error estimates follow based on the exact Kalman gain solution $\k = \sigma^2H$.
\subsection{Proof of Proposition \ref{prop:eps-infty}}
The proof relies on the fact that $\epsilon^{d/2}\geps(x,y)$ converges to a constant as $\epsilon \to \infty$. This would imply that $\keps(x,y)\to 1$ as $\epsilon \to \infty$. Therefore for a fixed function $f$, 
\begin{equation*}
\lim_{\epsilon \to \infty }\Teps f (x) = \int f(x)\pr(x)\ud x =: \hat{f} 
\end{equation*}
Define the limit $T_\infty:=\lim_{\epsilon\to\infty}\Teps$ and observe: 
\begin{equation*}
\lim_{\epsilon\to \infty} \frac{\phieps}{\epsilon} = \lim_{\epsilon \to \infty}(I-\Teps)^{-1}(h-\hat{h}) =(I-T_\infty)^{-1}(h-\hat{h})=h-\hat{h}
\end{equation*}
where the last step uses the fact that $\hat{h}=T_\infty h$ and we assumed $(I-\Teps)^{-1}h\to (I-T_\infty)h$. Then the gain approximation formula \eqref{eq:kepsN-approx} implies:
\begin{equation*}
\begin{aligned}
\lim_{\epsilon \to \infty} \k_\epsilon(x)&=\lim_{\epsilon \to \infty} \frac{1}{2}\left[\Teps (e\frac{\phieps}{\epsilon}) - \Teps(e)\Teps(\frac{\phieps}{\epsilon})\right] \\
&+\lim_{\epsilon \to \infty}\frac{1}{2}\left[\Teps(eh) - \Teps(e)\Teps(h)\right] \\
&=\int x(h(x)-\hat{h})\rho(x)\ud x 
\end{aligned}
\end{equation*} 
The argument for the finite-$N$ case is identical and omitted on
account of space.
\end{document}